\documentclass[11pt]{article}

\usepackage[T1]{fontenc}
\usepackage[utf8]{inputenc}
\usepackage{latexsym,bm}
\usepackage{mathrsfs}
\usepackage{amsmath,amssymb,amsthm,mathtools}
\usepackage{graphicx}
\usepackage{color}
\usepackage{cite}
\usepackage{booktabs,array}
\usepackage{stfloats}
\usepackage{caption}
\usepackage{epstopdf}
\usepackage{upgreek}
\usepackage{microtype}
\usepackage{tikz}
\usetikzlibrary{calc}
\usepackage[colorlinks,
            linkcolor=blue,
            anchorcolor=blue,
            citecolor=blue]{hyperref}

\theoremstyle{plain}
\newtheorem{theorem}{Theorem}[section]

\newtheorem{lemma}[theorem]{Lemma}
\newtheorem{proposition}[theorem]{Proposition}
\newtheorem{corollary}[theorem]{Corollary}
\newtheorem{conjecture}[theorem]{Conjecture}
\theoremstyle{remark}
\newtheorem{remark}[theorem]{Remark}
\theoremstyle{definition}
\newtheorem{definition}[theorem]{Definition}

\definecolor{unitcolor}{RGB}{30,30,30}
\definecolor{longcolor}{RGB}{0,114,178}
\definecolor{matchcolor}{RGB}{213,94,0}

\DeclareGraphicsExtensions{.pdf,.png,.jpg,.jpeg,.eps,.eps.gz}
\usepackage[top=2.5cm,bottom=2.5cm,left=2.8cm,right=2.2cm]{geometry}
\allowdisplaybreaks[4]
\makeatletter
\@addtoreset{equation}{section}
\makeatother

\usepackage{indentfirst}
\newcommand{\LLY}{Lin--Lu--Yau}

\newcommand{\kzero}{\kappa_0}
\newcommand{\RF}{RF^5_{72}}

\begin{document}

\begin{center}
{{\huge On the Lei--Bai conjecture on $5$-regular Lin--Lu--Yau Ricci-flat graphs}}\\[18pt]
{\Large Guangfu Wang$^{1,2}$, Wensheng Sun$^{3}$, and Yujun Yang$^{4}$\footnotemark[1]}\\[6pt]
{\footnotesize
$^{1}$ School of Mathematics and Information Sciences, Yantai University, Yantai, Shandong, 264005, China\\
$^{2}$ ECOPRO, Institute for Basic Science,  55 Expo-ro, Yuseong-gu, Daejeon, 34126, Korea\\
$^{3}$ School of Mathematics and Statistics, Gansu Center for Applied Mathematics, Lanzhou University, Lanzhou, Gansu 730000, China\\
$^{4}$ Department of Mathematics and Artificial Intelligence, Qilu University of Technology (Shandong Academy of Sciences), Jinan, Shandong, 250100, China\\
E-mail addresses: \texttt{gfwang@ytu.edu.cn}, \texttt{wensheng07002@163.com}, and \texttt{yangyujun@qlu.edu.cn}}
\end{center}
\footnotetext[1]{Corresponding author. E-mail address: \texttt{yangyujun@qlu.edu.cn}.}

\vspace{1mm}
\begin{abstract}

We study the Ricci curvature introduced by Lin, Lu, and Yau. A graph is called Ricci-flat if every edge has curvature zero. Lei and Bai classified $5$-regular symmetric Ricci-flat graphs by proving that every such graph is isomorphic to a particular $72$-vertex graph $\RF$, and conjectured that every $5$-regular Ricci-flat graph is either isomorphic to $\RF$ or admits a nontrivial Cartesian product decomposition. In this paper, we disprove this conjecture by constructing an infinite family of connected $5$-regular Ricci-flat graphs, none of which is isomorphic to $\RF$ or admits a nontrivial Cartesian product decomposition. This shows that the conjectured extension of the classification from the symmetric setting to general $5$-regular Ricci-flat graphs fails and that the class of such graphs is substantially richer than previously conjectured. To establish these results, we use an optimal-assignment formulation of Lin--Lu--Yau curvature to verify the Ricci-flatness of the constructed graphs.

\vspace{1mm}
\end{abstract}

\section{Introduction}

Ricci curvature is a fundamental concept in differential geometry and geometric analysis, relating the local geometry of a space to many of its global geometric and analytic properties. Its importance in the smooth setting has motivated the development of various discrete notions of curvature, which have been defined and exploited to understand geometric, spectral, and diffusion properties of graphs. These include, among others, Bakry--\'{E}mery curvature \cite{dcu}, Forman curvature \cite{Forman}, Ollivier--Ricci curvature \cite{Ollivier}, and Lin--Lu--Yau curvature \cite{LLY}. More recent examples include Steinerberger curvature \cite{sst}, defined in terms of equilibrium measures and the graph distance matrix, and node and edge resistance curvatures \cite{kde1,kde}, defined in terms of effective resistance.

Among these discrete curvature notions, Ollivier--Ricci curvature has been extensively studied. In 2009, Ollivier \cite{Ollivier} introduced a coarse Ricci curvature for Markov chains on metric spaces using optimal transport. In the graph setting, Ollivier--Ricci curvature compares the distance between adjacent vertices with the optimal transportation distance between probability measures supported on their neighborhoods, and thus reflects the local geometric and connectivity structure around an edge. Subsequently, Lin, Lu, and Yau \cite{LLY} introduced an idleness-independent modification of Ollivier--Ricci curvature by considering its derivative as the idleness parameter tends to one. The resulting Lin--Lu--Yau curvature has since received considerable attention in graph theory. In this paper, we follow the Lin--Lu--Yau definition of Ricci curvature on graphs.

A particularly natural class arising from the Lin--Lu--Yau curvature is that of Ricci-flat graphs \cite{LinLuYauGirth}, namely, graphs whose curvature vanishes on every edge. This terminology is motivated by Ricci-flat manifolds in differential geometry, which play an important role in geometry and mathematical physics \cite{BesseEinstein} and include notable examples such as Calabi--Yau manifolds \cite{YauCalabi}. The analogy has led to considerable interest in understanding the structure and classification of Ricci-flat graphs. The classification of Ricci-flat graphs has been studied under various degree and girth restrictions. Lin, Lu, and Yau \cite{LinLuYauGirth} classified Ricci-flat graphs with girth at least five, while Cushing, Kangaslampi, Lin, Liu, Lu, and Yau \cite{CushingCubic} later corrected and refined the cubic girth-five case. He, Luo, Yang, Yuan, and Zhang \cite{HeEtAl} constructed an infinite family of Ricci-flat graphs of girth four with edge-disjoint $4$-cycles and completely characterized those with vertex-disjoint $4$-cycles. Bai, Lu, and Yau \cite{BaiLuYau} classified all Ricci-flat graphs with maximum degree at most four.

Recall that a graph is called symmetric if its automorphism group acts transitively on both its vertex set and its edge set. Along these lines, Lei and Bai studied $5$-regular symmetric Ricci-flat graphs. They constructed a $5$-regular symmetric Ricci-flat graph $\RF$ on $72$ vertices and proved that it is unique up to isomorphism \cite{LeiBai}. Motivated by this result, they proposed the following conjecture as a generalization from the symmetric case to all $5$-regular Ricci-flat graphs. In this paper, a graph is said to be of \emph{Cartesian product type} if it is isomorphic to a nontrivial Cartesian product $X\square Y$ of finite simple graphs, where both $X$ and $Y$ have at least two vertices.

\begin{conjecture}[Lei and Bai \cite{LeiBai}]\label{conj:Lei-Bai}
Every $5$-regular Ricci-flat graph is either isomorphic to $\RF$ or is of Cartesian product type.
\end{conjecture}

In this paper, we focus on this conjecture and disprove it by constructing an infinite family of connected $5$-regular Ricci-flat graphs, none of which is isomorphic to $\RF$ or admits a nontrivial Cartesian product decomposition. Our main result is the following.

\begin{theorem}\label{thm:main}
There exists an infinite family of connected $5$-regular Ricci-flat graphs, each of which is neither isomorphic to $\RF$ nor of Cartesian product type.
\end{theorem}

The remainder of this paper is organized as follows. In Section~2, we review the necessary definitions and preliminary results concerning Ollivier--Ricci curvature, Lin--Lu--Yau curvature, optimal assignments, and Cartesian products of graphs. In Section~3, we construct an infinite family of connected $5$-regular Ricci-flat graphs that are neither isomorphic to $\RF$ nor of Cartesian product type, thereby proving the main result. In Section~4, we further study the zero-idleness Ollivier--Ricci curvature of the constructed graphs and show that they are not bone-idle.
\section{Preliminaries}

All graphs considered in this paper are finite, simple, connected, and unweighted. Let $G=(V,E)$ be a graph. For each vertex $x\in V$, let $N(x)$ denote its neighborhood, let $N[x]=N(x)\cup{x}$ denote its closed neighborhood, and let $d_x=|N(x)|$ denote its degree. Two vertices $x,y\in V$ are said to be adjacent if $xy\in E$, in which case we write $x\sim y$. The distance between two vertices $u,v\in V$, denoted by $d(u,v)$, is the length of a shortest $u$--$v$ path in $G$. For two graphs $G$ and $H$, we write $G\cong H$ to mean that $G$ and $H$ are isomorphic.

A probability measure on $V$ is a function $\mu:V\to[0,1]$ satisfying $\sum_{z\in V}\mu(z)=1$. For a vertex $x\in V$ and a parameter $\alpha\in[0,1]$, define the probability measure $\mu_x^\alpha$ on $V$ by
\[
\mu_x^\alpha(z)=
\begin{cases}
\alpha, & z=x,\\[1mm]
\dfrac{1-\alpha}{d_x}, & z\in N(x),\\[2mm]
0, & \text{otherwise}.
\end{cases}
\]
Indeed, we have
\[
\sum_{z\in V}\mu_x^\alpha(z)
=\alpha+\sum_{z\in N(x)}\frac{1-\alpha}{d_x}
=\alpha+d_x\frac{1-\alpha}{d_x}
=1.
\]
Thus, a mass $\alpha$ is placed at $x$, while the remaining mass $1-\alpha$ is distributed uniformly among the neighbors of $x$. The parameter $\alpha$ is called the \emph{idleness}.

\begin{definition}[1-Wasserstein distance]\label{def:Wasserstein}
Let $\mu$ and $\nu$ be two probability measures on $V$. A transport plan from $\mu$ to $\nu$ is a function $\pi:V\times V\to[0,1]$ satisfying
\[
\sum_{v\in V}\pi(u,v)=\mu(u)\quad\text{for every }u\in V,
\qquad
\sum_{u\in V}\pi(u,v)=\nu(v)\quad\text{for every }v\in V.
\]
We denote the set of all transport plans from $\mu$ to $\nu$ by $\Pi(\mu,\nu)$. The $1$-Wasserstein distance between $\mu$ and $\nu$, with respect to the graph metric, is defined by
\[
W_1(\mu,\nu)=\min_{\pi\in\Pi(\mu,\nu)}\sum_{u,v\in V}d(u,v)\pi(u,v).
\]
\end{definition}

\begin{definition}[Ollivier--Ricci curvature \cite{Ollivier}]\label{def:OR-curvature}
Let $G=(V,E)$ be a connected graph. For two distinct vertices $x,y\in V$ and $\alpha\in[0,1]$, the Ollivier--Ricci curvature with idleness $\alpha$ is defined by
\[
\kappa_\alpha(x,y)=1-\frac{W_1(\mu_x^\alpha,\mu_y^\alpha)}{d(x,y)}.
\]
In particular, if $x\sim y$, then $d(x,y)=1$, and hence $\kappa_\alpha(x,y)=1-W_1(\mu_x^\alpha,\mu_y^\alpha)$.
\end{definition}

\begin{definition}[Lin--Lu--Yau curvature \cite{LLY}]\label{def:LLY-curvature}
Let $G=(V,E)$ be a connected graph. The Lin--Lu--Yau curvature of an edge $xy$ is defined by
\begin{equation}\label{eq:LLY-definition}
\kappa(x,y)=\lim_{\alpha\to1^-}\frac{\kappa_\alpha(x,y)}{1-\alpha}
=\lim_{\alpha\to1^-}\frac{1-W_1(\mu_x^\alpha,\mu_y^\alpha)}{1-\alpha}.
\end{equation}
\end{definition}

\begin{definition}[Ricci-flat graph \cite{HehlRegular}]\label{def:Ricci-flat}
A graph $G$ is called \emph{Ricci-flat} if $\kappa(x,y)=0$ for every edge $xy\in E(G)$.
\end{definition}

For regular graphs, the Lin--Lu--Yau curvature admits a convenient optimal-assignment formulation.

\begin{definition} \label{def:assignment}
Let $G=(V,E)$ be a connected graph. Let $x\sim y$ with $d_x=d_y=d$. Define $L_{xy}=N(x)\setminus N[y]$ and $R_{xy}=N(y)\setminus N[x]$. Let $c_{xy}=|N(x)\cap N(y)|$. Then $|L_{xy}|=|R_{xy}|=d-1-c_{xy}$. An \emph{assignment} from $L_{xy}$ to $R_{xy}$ is a bijection $\varphi:L_{xy}\to R_{xy}$. Let $\mathcal A_{xy}$ denote the set of all such assignments, and define
\[
m(x,y)=\min_{\varphi\in\mathcal A_{xy}}\sum_{u\in L_{xy}}d\bigl(u,\varphi(u)\bigr).
\]
When $L_{xy}=R_{xy}=\varnothing$, we use the convention $m(x,y)=0$.
\end{definition}

\begin{proposition}[{\cite{HehlRegular}}]\label{prop2.6}
Let $G=(V,E)$ be a connected graph, and let $xy\in E$ satisfy $d_x=d_y=d$. Then
\begin{equation}\label{eq:assignment-formula}
\kappa(x,y)=\frac{1}{d}\bigl(d+1-m(x,y)\bigr).
\end{equation}
\end{proposition}

\begin{remark}
At $\alpha=1/(d+1)$, the measures $\mu_x^\alpha$ and $\mu_y^\alpha$ are the uniform probability measures on $N[x]$ and $N[y]$, respectively. Their common support can be matched to itself at zero cost, and the remaining transportation problem reduces to a minimum-cost assignment from $L_{xy}$ to $R_{xy}$. Together with $\kappa(x,y)=\frac{d+1}{d}\kappa_{1/(d+1)}(x,y)$ \cite{bou}, this yields \eqref{eq:assignment-formula}.
\end{remark}

In particular, if $G$ is $5$-regular, then $\kappa(x,y)=0$ if and only if $m(x,y)=6$.

\begin{lemma}[\cite{LLY,HehlRegular}]\label{lem2.1}
For every $n\geq2$, the complete graph $K_n$ has constant Lin--Lu--Yau curvature $\kappa(K_n)=n/(n-1)$. Moreover, if a graph $G$ satisfies $\kappa(x,y)>1$ for every edge $xy\in E(G)$, then $G$ is a complete graph.
\end{lemma}

\begin{lemma}[\cite{LLY}]\label{lem2.2}
Let $C_n$ be the cycle of length $n\geq3$. Then every edge of $C_n$ has the same Lin--Lu--Yau curvature, with $\kappa(C_3)=3/2$, $\kappa(C_4)=1$, and $\kappa(C_5)=1/2$, whereas $\kappa(C_n)=0$ for every $n\geq6$. In particular, $C_n$ is Ricci-flat for every $n\geq6$.
\end{lemma}

The \emph{Cartesian product} of two graphs $G$ and $H$, denoted by $G\square H$, is the graph with vertex set $V(G)\times V(H)$ in which two vertices $(u,x)$ and $(v,y)$ are adjacent if and only if either $u=v$ and $xy\in E(H)$, or $x=y$ and $uv\in E(G)$. Thus, every edge of $G\square H$ changes exactly one coordinate. An edge joining $(u,x)$ and $(v,x)$, where $uv\in E(G)$, is called a $G$-direction edge, while an edge joining $(u,x)$ and $(u,y)$, where $xy\in E(H)$, is called an $H$-direction edge.

We shall also use the Cartesian product formula for Lin--Lu--Yau curvature. Let $X$ and $Y$ be regular graphs of degrees $d_X$ and $d_Y$, respectively. Then $X\square Y$ is $(d_X+d_Y)$-regular.

\begin{theorem}[\cite{CushingProducts,LLY}]\label{thm:product-formula}
Let $X$ and $Y$ be $d_X$-regular and $d_Y$-regular graphs, respectively. For an $X$-direction edge $(x,y)(x_1,y)$ of $X\square Y$, where $xx_1\in E(X)$, we have
\begin{equation}\label{eq:product-formula}
\kappa^{X\square Y}\bigl((x,y),(x_1,y)\bigr)=\frac{d_X}{d_X+d_Y}\kappa^X(x,x_1).
\end{equation}
Similarly, for a $Y$-direction edge $(x,y)(x,y_1)$ of $X\square Y$, where $yy_1\in E(Y)$, we have
\begin{equation}\label{eq:product-formula-Y}
\kappa^{X\square Y}\bigl((x,y),(x,y_1)\bigr)=\frac{d_Y}{d_X+d_Y}\kappa^Y(y,y_1).
\end{equation}
\end{theorem}

\section{Main results}\label{sec:main}

We now define the family of graphs used to prove Theorem~\ref{thm:main}. Let $q\geq5$ be an integer and set $n=5q$. Throughout this section, all subscripts are taken modulo $n$, and we write $\mathbb Z_n=\{0,1,\ldots,n-1\}$. Define a graph $H_q$ with vertex set
$$V(H_q)=\{a_i,b_i:i\in\mathbb Z_n\}.$$
Its edges are determined by the following adjacency relations:
\begin{equation}\label{eq:Hqdef}
a_i\sim a_{i\pm1},\ a_{i\pm q},\ b_i,
\qquad
b_i\sim b_{i\pm1},\ b_{i\pm2q},\ a_i
\end{equation}
for every $i\in\mathbb Z_n$.

According to the construction,  the graph $H_q$ may be viewed as two circulant layers indexed by $\mathbb Z_n$. The $a$-layer has connection set $\{\pm1,\pm q\}$, whereas the $b$-layer has connection set $\{\pm1,\pm2q\}$, and corresponding vertices in the two layers are joined by the perfect matching $\{a_ib_i:i\in\mathbb Z_n\}$. Since $n=5q$ and $q\geq5$, the four neighbors of each vertex within its own layer are distinct. For illustration, the graph $H_5$ is shown in Figure~\ref{fig:H5}. Clearly, $H_q$ has $2n=10q$ vertices and is $5$-regular. In each layer, the edges joining vertices whose indices differ by $1$ form a spanning cycle of length $n$. Thus both layers are connected, and the matching between them implies that $H_q$ is connected.

\begin{figure}[htbp]
\centering
\resizebox{0.82\textwidth}{!}{%
\begin{tikzpicture}[
    x=1cm,
    y=1cm,
    vertex/.style={circle,fill=black,inner sep=1.55pt},
    vlabel/.style={font=\scriptsize,fill=white,inner sep=0.65pt},
    cycleedge/.style={
        draw=unitcolor,
        line width=0.95pt,
        line cap=round
    },
    longedge/.style={
        draw=longcolor,
        line width=0.70pt,
        line cap=round
    },
    matchedge/.style={
        draw=matchcolor,
        line width=0.82pt,
        line cap=round
    },
    outermatch/.style={
        draw=matchcolor,
        line width=0.82pt,
        line cap=round,
        line join=round
    }
]

\def\R{3.25}
\coordinate (CA) at (-5.75,0);
\coordinate (CB) at ( 5.75,0);

\foreach \i in {0,...,24}{
    \pgfmathsetmacro{\anga}{90-14.4*\i}
    \pgfmathsetmacro{\angb}{90+14.4*\i}
    \coordinate (a\i) at ($(CA)+(\anga:\R)$);
    \coordinate (b\i) at ($(CB)+(\angb:\R)$);
}

\foreach \i in {0,...,24}{
    \pgfmathtruncatemacro{\ja}{mod(\i+5,25)}
    \pgfmathtruncatemacro{\jb}{mod(\i+10,25)}
    \draw[longedge] (a\i) -- (a\ja);
    \draw[longedge] (b\i) -- (b\jb);
}

\foreach \i in {0,...,24}{
    \pgfmathtruncatemacro{\j}{mod(\i+1,25)}
    \draw[cycleedge] (a\i) -- (a\j);
    \draw[cycleedge] (b\i) -- (b\j);
}

\foreach \i in {0,...,24}{
    \node[vertex] at (a\i) {};
    \node[vertex] at (b\i) {};
}

\foreach \i in {0,...,24}{
    \pgfmathsetmacro{\anga}{90-14.4*\i}
    \pgfmathsetmacro{\angb}{90+14.4*\i}
    \node[vlabel] at ($(CA)+(\anga:3.69)$) {$a_{\i}$};
    \node[vlabel] at ($(CB)+(\angb:3.69)$) {$b_{\i}$};
}

\foreach \i in {0,...,12}{
    \draw[matchedge] (a\i) -- (b\i);
}

\foreach \i in {13,...,18}{
    \pgfmathsetmacro{\h}{-4.15-0.18*(\i-13)}
    \draw[outermatch]
        (a\i)
        .. controls (-8.80,\h) and (8.80,\h) ..
        (b\i);
}

\foreach \i in {19,...,24}{
    \pgfmathsetmacro{\h}{4.15+0.18*(24-\i)}
    \draw[outermatch]
        (a\i)
        .. controls (-8.80,\h) and (8.80,\h) ..
        (b\i);
}

\begin{scope}[shift={(-4.10,-5.50)}]
    \draw[cycleedge] (0,0) -- (0.85,0);
    \node[
        anchor=west,
        font=\scriptsize,
        text=unitcolor
    ] at (1.00,0) {unit edges};

    \draw[longedge] (3.15,0) -- (4.00,0);
    \node[
        anchor=west,
        font=\scriptsize,
        text=longcolor
    ] at (4.15,0) {long-layer edges};

    \draw[matchedge] (7.40,0) -- (8.25,0);
    \node[
        anchor=west,
        font=\scriptsize,
        text=matchcolor
    ] at (8.40,0) {matching edges};
\end{scope}

\end{tikzpicture}%
}
\caption{The graph $H_5$. The black edges are the unit edges $a_i a_{i+1}$ and $b_i b_{i+1}$, the blue edges are the long-layer edges $a_i a_{i+5}$ and $b_i b_{i+10}$, and the orange-red edges are the matching edges $a_i b_i$. All subscripts are taken modulo $25$.}
\label{fig:H5}
\end{figure}
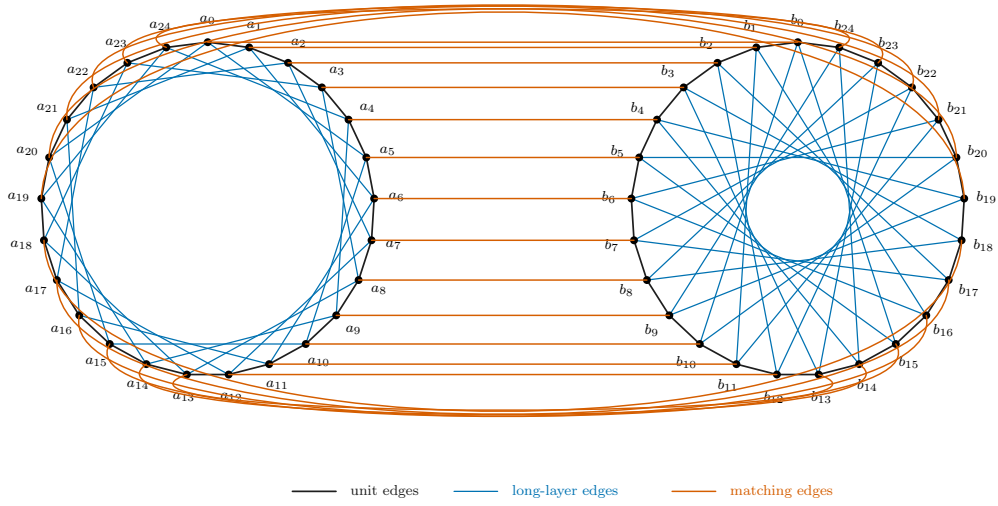

For each $t\in\mathbb Z_n$, the map $\tau_t$ defined by
\[
\tau_t(a_i)=a_{i+t}\qquad\text{and}\qquad \tau_t(b_i)=b_{i+t}
\]
is an automorphism of $H_q$. Consequently, these translation automorphisms partition $E(H_q)$ into five edge orbits, for which we may choose the representative set
\begin{equation}\label{eq:Hqorbits}
\mathcal R=\{a_0a_1,\ a_0a_q,\ b_0b_1,\ b_0b_{2q},\ a_0b_0\}.
\end{equation}

\begin{theorem}\label{thm:Hqflat}
For every $q\ge5$, the graph $H_q$ is Ricci-flat with respect to \LLY\ curvature.
\end{theorem}

\begin{proof}
As shown above, the graph $H_q$ is $5$-regular. To compute the Lin--Lu--Yau curvature, it suffices to consider the five representative edges in $\mathcal R$ given by \eqref{eq:Hqorbits}. Table~\ref{tab:Hqsets} lists the corresponding sets $L_{xy}$ and $R_{xy}$, while Table~\ref{tab:Hqmatrices} gives the associated distance matrices. We now justify the entries of these matrices.
\begin{table}[ht]
\centering
\caption{Ordered sets $L_{xy}$ and $R_{xy}$ for the five representative edge types of $H_q$.}
\label{tab:Hqsets}
\small
\setlength{\tabcolsep}{3pt}
\begin{tabular}{c>{\raggedright\arraybackslash}p{0.38\textwidth}>{\raggedright\arraybackslash}p{0.38\textwidth}}
\toprule
edge $xy \in \mathcal R$ & $L_{xy}$ & $R_{xy}$ \\
\midrule
$a_0a_1$ & $(a_{-1},a_{-q},a_q,b_0)$ & $(a_2,a_{1-q},a_{1+q},b_1)$ \\
$a_0a_q$ & $(a_{-1},a_1,a_{-q},b_0)$ & $(a_{2q},a_{q-1},a_{q+1},b_q)$ \\
$b_0b_1$ & $(a_0,b_{-1},b_{-2q},b_{2q})$ & $(a_1,b_2,b_{1-2q},b_{1+2q})$ \\
$b_0b_{2q}$ & $(a_0,b_{-1},b_1,b_{-2q})$ & $(b_{2q-1},b_{2q+1},a_{2q},b_{4q})$ \\
$a_0b_0$ & $(a_{-1},a_1,a_{-q},a_q)$ & $(b_{-1},b_1,b_{-2q},b_{2q})$ \\
\bottomrule
\end{tabular}
\end{table}

Recall that $n=5q$ and that all index differences are taken in $\mathbb Z_n$. Define
\begin{align*}
\Delta_{aa}^{(2)}
&=\{0,\pm1,\pm2,\pm q,\pm(q-1),\pm(q+1),\pm2q\},\\
\Delta_{bb}^{(2)}
&=\{0,\pm1,\pm2,\pm2q,\pm(2q-1),\pm(2q+1),\pm4q\},\\
\Delta_{ab}^{(2)}
&=\{0,\pm1,\pm q,\pm2q\}.
\end{align*}
Since each edge within the $a$-layer changes the index by $\pm1$ or $\pm q$, a path of length at most two between $a_i$ and $a_j$ exists if and only if $j-i\in\Delta_{aa}^{(2)}$. Similarly, each edge within the $b$-layer changes the index by $\pm1$ or $\pm2q$, and hence a path of length at most two between $b_i$ and $b_j$ exists if and only if $j-i\in\Delta_{bb}^{(2)}$. Finally, a path of length at most two between vertices in different layers must use exactly one matching edge and at most one edge within a layer. Therefore, such a path between $a_i$ and $b_j$ exists if and only if $j-i\in\Delta_{ab}^{(2)}$. Consequently,
\begin{align*}
d(a_i,a_j)\leq2&\iff j-i\in\Delta_{aa}^{(2)},\\
d(b_i,b_j)\leq2&\iff j-i\in\Delta_{bb}^{(2)},\\
d(a_i,b_j)\leq2&\iff j-i\in\Delta_{ab}^{(2)}.
\end{align*}

The preceding equivalences determine whether the distance between a vertex of $L_{xy}$ and a vertex of $R_{xy}$ is at most two; in that case, the adjacency relations distinguish distance one from distance two. On the other hand, for every $u\in L_{xy}$ and $v\in R_{xy}$, the walk $u-x-y-v$ has length three. It follows that if the relevant index difference does not belong to the appropriate set above, then $d(u,v)=3$. Therefore, applying these criteria to the ordered sets listed in Table~\ref{tab:Hqsets} gives the distance matrices displayed in Table~\ref{tab:Hqmatrices}.

\begin{table}[htbp]
\centering
\caption{Distance matrices $D_{xy}$ for the five representative edge types of $H_q$.}
\label{tab:Hqmatrices}
\small
\setlength{\arraycolsep}{3pt}
\setlength{\tabcolsep}{14pt}
\begin{tabular}{@{}cc@{}}
$\displaystyle
D_{a_0a_1}=\begin{pmatrix}
3&3&3&3\\
3&1&3&3\\
3&3&1&3\\
3&3&3&1
\end{pmatrix}$
&
$\displaystyle
D_{a_0a_q}=\begin{pmatrix}
3&1&3&3\\
3&3&1&3\\
2&3&3&2\\
2&3&3&2
\end{pmatrix}$
\\[1.8em]
$\displaystyle
D_{b_0b_1}=\begin{pmatrix}
1&3&3&3\\
3&3&3&3\\
3&3&1&3\\
3&3&3&1
\end{pmatrix}$
&
$\displaystyle
D_{b_0b_{2q}}=\begin{pmatrix}
3&3&2&2\\
1&3&3&3\\
3&1&3&3\\
3&3&2&2
\end{pmatrix}$
\\[1.8em]
\multicolumn{2}{c}{$\displaystyle
D_{a_0b_0}=\begin{pmatrix}
1&3&3&3\\
3&1&3&3\\
3&3&2&2\\
3&3&2&2
\end{pmatrix}$}
\end{tabular}
\end{table}
For a representative edge $xy$, write $L_{xy}=\{\ell_1,\ell_2,\ell_3,\ell_4\}$ and $R_{xy}=\{r_1,r_2,r_3,r_4\}$ in the orders given in Table~\ref{tab:Hqsets}. Let $\mathcal B_{xy}$ be the weighted complete bipartite graph with bipartition $(L_{xy},R_{xy})$, where the edge $\ell_i r_j$ has weight $w(\ell_i r_j)=d(\ell_i,r_j)$ for $1\leq i,j\leq4$. Thus, $D_{xy}$ is precisely the weighted biadjacency matrix of $\mathcal B_{xy}$. Every bijection from $L_{xy}$ to $R_{xy}$ corresponds uniquely to a perfect matching of $\mathcal B_{xy}$, and conversely. Therefore,
\[
m(x,y)=\min\{w(M):M\text{ is a perfect matching of }\mathcal B_{xy}\},
\]
where $w(M)$ denotes the sum of the weights of the four edges in $M$.

For each representative edge $xy$, Table~\ref{tab:Hqmatchings} gives a perfect matching of $\mathcal B_{xy}$ having total weight $6$. It follows that $m(x,y)\leq6$ for every representative edge.

\begin{table}[htbp]
\centering
\caption{Perfect matchings of weight $6$ in the five auxiliary complete bipartite graphs.}
\label{tab:Hqmatchings}
\small
\setlength{\tabcolsep}{4pt}
\begin{tabular}{c>{\raggedright\arraybackslash}p{0.56\textwidth}c}
\toprule
edge $xy$ & perfect matching & edge weights \\
\midrule
$a_0a_1$ & $\{\ell_1r_1,\ell_2r_2,\ell_3r_3,\ell_4r_4\}$ & $(3,1,1,1)$ \\
$a_0a_q$ & $\{\ell_1r_2,\ell_2r_3,\ell_3r_1,\ell_4r_4\}$ & $(1,1,2,2)$ \\
$b_0b_1$ & $\{\ell_1r_1,\ell_2r_2,\ell_3r_3,\ell_4r_4\}$ & $(1,3,1,1)$ \\
$b_0b_{2q}$ & $\{\ell_1r_3,\ell_2r_1,\ell_3r_2,\ell_4r_4\}$ & $(2,1,1,2)$ \\
$a_0b_0$ & $\{\ell_1r_1,\ell_2r_2,\ell_3r_3,\ell_4r_4\}$ & $(1,1,2,2)$ \\
\bottomrule
\end{tabular}
\end{table}

To show that $m(x,y)=6$ for every representative edge, it remains to prove that $m(x,y)\geq 6$ holds for each representative edge. Let $F_{xy}$ be the spanning subgraph of $\mathcal B_{xy}$ consisting of the edges of weight $1$. Since every edge outside $F_{xy}$ has weight at least $2$, any perfect matching containing at most two edges of $F_{xy}$ has total weight at least 6. Thus, it suffices to consider perfect matchings containing at least three edges of weight $1$.

For $xy=a_0a_1$, the only weight-$1$ edges are $\ell_2r_2$, $\ell_3r_3$, and $\ell_4r_4$. Hence any perfect matching containing three weight-$1$ edges must contain all three of them, and its remaining edge must be $\ell_1r_1$, which has weight $3$. Therefore, every perfect matching of $\mathcal B_{a_0a_1}$ has weight at least $6$.

For $xy=b_0b_1$, the weight-$1$ edges are $\ell_1r_1, \ell_3r_3$, and $\ell_4r_4$. Thus any perfect matching containing three weight-$1$ edges must also contain $\ell_2r_2$, whose weight is $3$. Consequently, every perfect matching of $\mathcal B_{b_0b_1}$ has weight at least $6$.

For each of the remaining representative edges $a_0a_q$, $b_0b_{2q}$, and $a_0b_0$, the graph $F_{xy}$ contains exactly two edges. Hence no perfect matching of the corresponding auxiliary graph can contain three or four edges of weight $1$. It follows that every perfect matching in each of these three cases also has weight at least $6$.

We have therefore proved that $m(x,y)\geq6$ for each representative edge. Together with Table~\ref{tab:Hqmatchings}, this gives $m(x,y)=6$ for all five representative edges. By Proposition \ref{prop2.6}, each representative edge has Lin--Lu--Yau curvature zero. Since every edge of $H_q$ is mapped to one of these five representative edges by a translation automorphism and Lin--Lu--Yau curvature is invariant under graph automorphisms, we conclude that $\kappa(x,y)=0$ for every edge $xy\in E(H_q)$. Therefore, $H_q$ is Lin--Lu--Yau Ricci-flat. This completes the proof.
\end{proof}

\begin{lemma}\label{lem:HqPetersen}
Let $U_q=\{a_i a_{i+1},\,b_i b_{i+1}:i\in\mathbb Z_n\}$ be the set of unit edges of $H_q$, where $n=5q$. Then the spanning subgraph $H_q-U_q$ is the disjoint union of $q$ copies of the Petersen graph. Moreover, an edge of $H_q$ lies in a $5$-cycle if and only if it does not belong to $U_q$.
\end{lemma}

\begin{proof}
For each $\rho\in\mathbb Z_q$, define $V_\rho=\{a_{\rho+tq},b_{\rho+tq}:t\in\mathbb Z_5\}$. Then the sets $V_\rho$ form a partition of $V(H_q)$. After the unit edges are deleted, the remaining edges are the $a$-layer edges of step $q$, the $b$-layer edges of step $2q$, and the matching edges. Each such edge joins two vertices whose subscripts are congruent modulo $q$. Hence there are no edges of $H_q-U_q$ between two distinct sets $V_\rho$ and $V_{\rho'}$. Fix $\rho\in\mathbb Z_q$ and write $u_t=a_{\rho+tq}$ and $v_t=b_{\rho+tq}$ for $t\in\mathbb Z_5$. Then the subgraph induced by $V_\rho$ has edges $u_tu_{t\pm1}$, $v_tv_{t\pm2}$, and $u_tv_t$ for $t\in\mathbb Z_5$. Thus $u_0u_1u_2u_3u_4u_0$ and $v_0v_2v_4v_1v_3v_0$ are $5$-cycles, while the edges $u_tv_t$, $t\in\mathbb Z_5$, form a perfect matching between them. This is the standard representation of the Petersen graph. Therefore, the subgraph induced by $V_\rho$ is isomorphic to the Petersen graph, and consequently $H_q-U_q$ is the disjoint union of $q$ copies of the Petersen graph.

We next show that every edge outside $U_q$ lies in a $5$-cycle. For each $i\in\mathbb Z_n$, the edge $a_ia_{i+q}$ lies in the $5$-cycle $a_ia_{i+q}a_{i+2q}a_{i+3q}a_{i+4q}a_i$, the edge $b_ib_{i+2q}$ lies in the $5$-cycle $b_ib_{i+2q}b_{i+4q}b_{i+q}b_{i+3q}b_i$, and the matching edge $a_ib_i$ lies in the $5$-cycle $a_ib_ib_{i+2q}a_{i+2q}a_{i+q}a_i$. Therefore, every edge outside $U_q$ lies in a $5$-cycle.

It remains to prove that no unit edge lies in a $5$-cycle. By translation invariance, it suffices to consider $a_0a_1$ and $b_0b_1$. If a $5$-cycle contains one of these edges, then deleting that edge leaves a path of length four between its endpoints. Since the endpoints lie in the same layer, this path contains zero, two, or four matching edges. First consider the edge $a_0a_1$ and distinguish the following cases.

\textbf{Case 1.} Suppose that the path lies entirely in the $a$-layer. Then each edge changes the subscript by $+1$, $-1$, $+q$, or $-q$. Let $p$ and $m$ be the numbers of edges that change the subscript by $+q$ and $-q$, respectively, and set $k=p+m$ and $t=p-m$. Thus $k$ is the number of long edges in the path, and these long edges produce a net change of $tq$ in the subscript. Let $u$ be the net change in the subscript produced by the remaining $4-k$ unit edges. Then $|t|\leq k$, $|u|\leq4-k$, and, since the path starts at $a_0$ and ends at $a_1$, we have $tq+u\equiv1\pmod{5q}$. Moreover, $-4q\leq tq+u\leq4q$. Since $q\geq5$, we have $tq+u=1$. If $t\ne0$, then $|t|\geq1$, so $|u|=|1-tq|\geq q-1\geq4$. On the other hand, $t\ne0$ implies $k\geq1$, and therefore $|u|\leq4-k\leq3$, a contradiction. If $t=0$, then $p=m$, so $k=2p$ is even. Hence the number $4-k$ of unit edges is even, and their net change $u$, being the sum of an even number of terms from $\{-1,1\}$, is also even. This contradicts $u=1$.

\textbf{Case 2.} Suppose that the path contains exactly two matching edges. Since matching edges do not change the subscript, the change from subscript $0$ to subscript $1$ must be produced by the remaining two non-matching edges. Each of these two edges changes the subscript by an element of $\{\pm1,\pm q,\pm2q\}$. Let their changes be $r$ and $s$. Then $r+s\equiv1\pmod{5q}$. If neither $r$ nor $s$ is equal to $\pm1$, then both are divisible by $q$, so $r+s\equiv0\pmod q$, contradicting $r+s\equiv1\pmod q$. If both belong to $\{\pm1\}$, then $r+s\in\{-2,0,2\}$, none of which is congruent to $1$ modulo $q$ because $q\geq5$. It remains to consider the case in which exactly one of $r$ and $s$ is a unit change. Without loss of generality, write the two changes as $\varepsilon$ and $cq$, where $\varepsilon\in\{-1,1\}$ and $c\in\{-2,-1,1,2\}$. Reducing $\varepsilon+cq\equiv1\pmod{5q}$ modulo $q$ gives $\varepsilon\equiv1\pmod q$, and hence $\varepsilon=1$. It follows that $cq\equiv0\pmod{5q}$, so $5\mid c$, which is impossible.

\textbf{Case 3.} Suppose that the path contains four matching edges. Then none of its edges changes the subscript, so it cannot join $a_0$ to $a_1$.

By the preceding three cases, there is no $a_0$--$a_1$ path of length $4$, and hence the edge $a_0a_1$ lies in no $5$-cycle.

Now consider the edge $b_0b_1$ and distinguish the following three cases.

\textbf{Case 1.} Suppose that there exists a $b_0$--$b_1$ path of length $4$ containing no matching edge. Then the path lies entirely in the $b$-layer, so each edge changes the subscript by an element of $\{\pm1,\pm2q\}$. Let $p$ and $m$ be the numbers of edges that change the subscript by $+2q$ and $-2q$, respectively, and set $k=p+m$ and $t=p-m$. Thus $k$ is the number of long edges, and these edges produce a net change of $2tq$ in the subscript. Let $u$ be the net change in the subscript produced by the remaining $4-k$ unit edges. Since $t=p-m$ and $k=p+m$, we have $|t|\leq k$. Moreover, $u$ is the sum of $4-k$ terms, each equal to $+1$ or $-1$, so $|u|\leq4-k$. Since the path starts at $b_0$ and ends at $b_1$, we have $2tq+u\equiv1\pmod{5q}$. Reducing this congruence modulo $q$ gives $u\equiv1\pmod q$.

\textbf{Subcase 1.1.} Suppose that $q\geq6$. Since $|u|\leq4$, the congruence $u\equiv1\pmod q$ forces $u=1$. Since $u$ is the sum of $4-k$ terms from $\{-1,1\}$, the integers $u$ and $4-k$ have the same parity. Thus $4-k$ is odd, and hence $k$ is odd. Moreover, $t=p-m$ and $k=p+m$ have the same parity, so $t$ is also odd. Since $0\leq k\leq4$, it follows that $k\in\{1,3\}$, and hence $0<|t|\leq k\leq3$. Substituting $u=1$ into $2tq+u\equiv1\pmod{5q}$ gives $2tq\equiv0\pmod{5q}$, and hence $5\mid2t$. Therefore, $5\mid t$, contradicting $0<|t|\leq3$.

\textbf{Subcase 1.2.} Suppose that $q=5$. In this case, $u\equiv1\pmod5$ and $|u|\leq4$, so $u\in\{1,-4\}$. If $u=1$, then the preceding parity argument shows that $k$ and $t$ are odd. Hence $k\in\{1,3\}$ and $0<|t|\leq3$. The congruence $10t+1\equiv1\pmod{25}$ implies $25\mid10t$, and therefore $5\mid t$, again contradicting $0<|t|\leq3$. If $u=-4$, then $u$ is the sum of $4-k$ terms from $\{-1,1\}$. The equality $|u|=4$ forces $k=0$, and all four unit edges change the subscript by $-1$. Consequently, $t=0$ and $2tq+u=-4\not\equiv1\pmod{25}$, again a contradiction. Therefore, no such path exists.

\textbf{Case 2.} Suppose that there exists a $b_0$--$b_1$ path of length $4$ containing exactly two matching edges. Matching edges do not change the subscript, so the change from subscript $0$ to subscript $1$ must be produced by the remaining two non-matching edges. Let the changes produced by these two edges be $r$ and $s$. Since each non-matching edge lies in either the $a$-layer or the $b$-layer, we have $r,s\in\{\pm1,\pm q,\pm2q\}$. Since the path joins $b_0$ to $b_1$, we also have $r+s\equiv1\pmod{5q}$.

If neither $r$ nor $s$ belongs to $\{\pm1\}$, then both are divisible by $q$, so $r+s\equiv0\pmod q$, contradicting $r+s\equiv1\pmod q$. If both $r$ and $s$ belong to $\{\pm1\}$, then $r+s\in\{-2,0,2\}$. Since $q\geq5$, none of these integers is congruent to $1$ modulo $q$, again giving a contradiction. Finally, suppose that exactly one of $r$ and $s$ belongs to $\{\pm1\}$. Write the two changes as $\varepsilon$ and $cq$, where $\varepsilon\in\{-1,1\}$ and $c\in\{-2,-1,1,2\}$. Reducing $\varepsilon+cq\equiv1\pmod{5q}$ modulo $q$ gives $\varepsilon\equiv1\pmod q$. Since $q\geq5$, this forces $\varepsilon=1$. It follows that $cq\equiv0\pmod{5q}$, so $5\mid c$, which is impossible because $c\in\{-2,-1,1,2\}$. Therefore, no such path exists.

\textbf{Case 3.} Suppose that there exists a $b_0$--$b_1$ path of length $4$ containing four matching edges. Then every edge of the path is a matching edge. Since matching edges do not change the subscript, such a path cannot join $b_0$ to $b_1$, a contradiction.

By the preceding three cases, there is no $b_0$--$b_1$ path of length $4$, and hence the edge $b_0b_1$ lies in no $5$-cycle.

Combining the above cases, we have proved that neither $a_0a_1$ nor $b_0b_1$ lies in a $5$-cycle. By translation invariance, no edge of $U_q$ lies in a $5$-cycle. Since every edge outside $U_q$ lies in a $5$-cycle, an edge of $H_q$ lies in a $5$-cycle if and only if it does not belong to $U_q$. This completes the proof.
\end{proof}

\begin{theorem}\label{thm:Hqnotproduct}
For every integer $q\geq5$, the graph $H_q$ is not isomorphic to $\RF$ and admits no nontrivial ordinary Cartesian product decomposition.
\end{theorem}

\begin{proof}
The graph $H_q$ has $2n=10q$ vertices. Since $10q\ne72$ for every integer $q\geq5$, it is not isomorphic to the $72$-vertex graph $\RF$.

It remains to prove that $H_q$ admits no nontrivial ordinary Cartesian product decomposition. Suppose, to the contrary, that $H_q\cong X_1\square\cdots\square X_s$, where $s\geq2$ and each $X_j$ is connected and nontrivial. For a vertex $(v_1,\ldots,v_s)$ of this product, its degree is $\sum_{j=1}^s d_{X_j}(v_j)$. Fix $j$ and fix all coordinates except the $j$-th one. Since $H_q$ is $5$-regular, comparing the degrees of the vertices obtained by choosing arbitrary $u_j,v_j\in V(X_j)$ gives $d_{X_j}(u_j)=d_{X_j}(v_j)$. Thus every factor $X_j$ is regular. Let $d_j\geq1$ denote the degree of $X_j$. Then $\sum_{j=1}^s d_j=5$.

Repeated application of the Cartesian product curvature formula \eqref{eq:product-formula} shows that, for an edge of $H_q$ arising from an edge $xy$ of the factor $X_j$, its curvature is $\frac{d_j}{5}\kappa^{X_j}(x,y)$. Since $H_q$ is Ricci-flat and $d_j>0$, every edge of every factor $X_j$ has curvature zero. Hence each $X_j$ is Ricci-flat. Moreover, every connected $1$-regular graph is isomorphic to $K_2$, whose unique edge has curvature $2$ by Lemma \ref{lem2.1}. Therefore, no factor can have degree $1$.  This means $d_j\geq2$ for every $j$. Since the positive integers $d_1,\ldots,d_s$ sum to $5$, it follows that $s=2$ and, after interchanging the factors if necessary, $d_1=2$ and $d_2=3$.

The finite connected $2$-regular factor is a cycle $C_m$ for some $m\geq3$. Since this factor is Ricci-flat, Lemma~\ref{lem2.2} gives $m\geq6$. Consequently, we may write $H_q\cong C_m\square Y$, where $Y$ is a finite connected $3$-regular Ricci-flat graph.

We first show that $Y$ is triangle-free. Suppose, for a contradiction, that an edge $xy\in E(Y)$ belongs to a triangle, so that $x$ and $y$ have a common neighbor. Since $Y$ is $3$-regular, we have $|L_{xy}|=|R_{xy}|=3-1-|N(x)\cap N(y)|\leq1$. If these sets are empty, then $m(x,y)=0$. If each contains one vertex, then the unique vertex of $L_{xy}$ can be joined to the unique vertex of $R_{xy}$ by the walk through $x$ and $y$ of length $3$, so $m(x,y)\leq3$. In either case, $m(x,y)\leq3$. However, Proposition~\ref{prop2.6} gives $\kappa(x,y)=\frac{1}{3}(4-m(x,y))>0$, contradicting the Ricci-flatness of $Y$. Hence $Y$ contains no triangle.

We next show that no $C_m$-direction edge of $C_m\square Y$ lies in a $5$-cycle. Suppose that a $5$-cycle contains a $C_m$-direction edge, and let $r$ be the number of its $C_m$-direction edges. Projecting the cycle onto $C_m$ and suppressing consecutive repetitions gives a closed walk of length $r$. Since $m\geq6$, the graph $C_m$ contains no odd cycle of length at most $5$, and hence no odd closed walk of length at most $5$. Thus $r$ is even. Since $1\leq r\leq5$, we have $r\in\{2,4\}$. If $r=2$, then the remaining three edges project onto a closed walk of length $3$ in $Y$, which forms a triangle, contrary to the triangle-freeness of $Y$. If $r=4$, then the remaining edge projects onto a closed walk of length $1$ in $Y$, which would be a loop, impossible because $Y$ is simple. Therefore, every edge of $C_m\square Y$ lying in a $5$-cycle is a $Y$-direction edge.

By Lemma~\ref{lem:HqPetersen}, every edge in $E(H_q)\setminus U_q$ lies in a $5$-cycle. By the preceding argument, every such edge must be a $Y$-direction edge. Hence, under the assumed isomorphism $H_q\cong C_m\square Y$, we have $E(H_q)\setminus U_q\subseteq E_Y$, where $E_Y$ denotes the set of all $Y$-direction edges.

For each vertex of $C_m$, the corresponding $Y$-fiber contains exactly $|E(Y)|$ $Y$-direction edges. Therefore, since $Y$ is $3$-regular and $|V(C_m)||V(Y)|=|V(H_q)|=10q$, we have
\[
|E_Y|=|V(C_m)||E(Y)|=|V(C_m)|\frac{3|V(Y)|}{2}=\frac{3|V(H_q)|}{2}=15q.
\]
On the other hand, since $H_q$ is $5$-regular on $10q$ vertices, it has $25q$ edges. The set $U_q$ consists of the $5q$ unit edges in each of the two layers, and hence $|U_q|=10q$. It follows that $|E(H_q)\setminus U_q|=25q-10q=15q$. Therefore, $E(H_q)\setminus U_q=E_Y$. Since every edge of $C_m\square Y$ is either a $C_m$-direction edge or a $Y$-direction edge, the $C_m$-direction edges are precisely the edges in $U_q$.

For each $y\in V(Y)$, the vertices $\{(c,y):c\in V(C_m)\}$ together with the $C_m$-direction edges between them induce a copy of $C_m$. Hence the spanning subgraph of $C_m\square Y$ formed by all $C_m$-direction edges is the disjoint union of $|V(Y)|$ copies of $C_m$. On the other hand, the spanning subgraph of $H_q$ formed by the unit edges consists of exactly the two cycles
\[
a_0a_1\cdots a_{5q-1}a_0
\qquad\text{and}\qquad
b_0b_1\cdots b_{5q-1}b_0.
\]
Since the $C_m$-direction edges correspond precisely to the unit edges in $U_q$, these two spanning subgraphs are isomorphic. Comparing their numbers of connected components gives $|V(Y)|=2$, which is impossible since $Y$ is $3$-regular. Hence $H_q$ admits no nontrivial ordinary Cartesian product decomposition. This completes the proof.
\end{proof}

\textbf{Proof of Theorem~\ref{thm:main}.} The result follows immediately from Theorems~\ref{thm:Hqflat} and~\ref{thm:Hqnotproduct}. \qed

\section{Bone-idleness of \texorpdfstring{$H_q$}{Hq}}

We first recall the notion of bone-idleness. For an edge $xy$ in a graph and an idleness parameter $\alpha\in[0,1]$, let $\kappa_\alpha(x,y)$ denote the Ollivier--Ricci curvature of $xy$, as defined in Definition~\ref{def:OR-curvature}. The edge $xy$ is called \emph{bone-idle} if $\kappa_\alpha(x,y)=0$ for every  $\alpha\in [0,1]$. A graph is called bone-idle if each of its edges is bone-idle.

It is clear that bone-idleness is stronger than Lin--Lu--Yau Ricci-flatness. Hehl proved the following characterization.

\begin{theorem}[Hehl \cite{HehlBone}]\label{thm:bone-idle-characterization}
Let $G=(V,E)$ be a locally finite graph and let $xy\in E(G)$. Then the following statements are equivalent:
\begin{enumerate}
\item The edge $xy$ is bone-idle, that is, $\kappa_\alpha(x,y)=0$ for every $\alpha\in[0,1]$.
\item $\kappa_0(x,y)=\kappa(x,y)=0$.
\end{enumerate}
Consequently, $G$ is bone-idle if and only if it is both $0$-Ricci-flat and Lin--Lu--Yau Ricci-flat.
\end{theorem}

The study of regular bone-idle graphs constitutes another natural and interesting direction in the study of discrete Ricci curvature. In this direction, Hehl \cite{HehlBone} proved that no $3$-regular bone-idle graph exists and gave a complete classification of all $4$-regular bone-idle graphs. In degree five, he proved that no symmetric $5$-regular graph is bone-idle and that no Cartesian product of a $3$-regular graph and a $2$-regular graph is bone-idle. The existence of a $5$-regular bone-idle graph remains open \cite{HehlBone}.

\begin{theorem}\label{thm:Hq-not-bone-idle}
Let $q\geq5$, let $n=5q$, and let $H_q$ be the graph defined above. Then, for every $i\in\mathbb Z_n$, we have $\kzero(a_i,a_{i+1})=\kzero(b_i,b_{i+1})=0$, whereas $\kzero(a_i,a_{i+q})=\kzero(b_i,b_{i+2q})=\kzero(a_i,b_i)=-1/5$. Hence $H_q$ is not bone-idle.
\end{theorem}

\begin{proof}
By the cyclic symmetry of $H_q$, it suffices to compute the zero-idleness Ollivier--Ricci curvature for the five representative edge types $a_0a_1$, $a_0a_q$, $b_0b_1$, $b_0b_{2q}$, and $a_0b_0$. For an edge $xy$ of a $d$-regular graph, define
\[
m_0(x,y)=\min_{\psi:N(x)\to N(y)}\sum_{z\in N(x)}d\bigl(z,\psi(z)\bigr),
\]
where the minimum is taken over all bijections $\psi:N(x)\to N(y)$. Since $\mu_x^0$ and $\mu_y^0$ are the uniform probability measures on the $d$-element sets $N(x)$ and $N(y)$, respectively, it follows from \cite[Lemma~2.6]{HehlRegular} that $W_1(\mu_x^0,\mu_y^0)=m_0(x,y)/d$. Hence, for every edge $xy$ of $H_q$, we have $\kzero(x,y)=1-m_0(x,y)/5$.

By Table~\ref{tab:Hqsets}, the endpoints of each representative edge $xy$ of $H_q$ have no common neighbor, and hence every summand in the definition of $m_0(x,y)$ is at least one. Therefore, $m_0(x,y)\geq5$. Suppose that the corresponding row of Table~\ref{tab:Hqsets} is written as $L_{xy}=(u_1,u_2,u_3,u_4)$ and $R_{xy}=(v_1,v_2,v_3,v_4)$. For convenience, set $u_5=y$ and $v_5=x$. Then
\[
N(x)=(u_1,u_2,u_3,u_4,u_5)
\qquad\text{and}\qquad
N(y)=(v_1,v_2,v_3,v_4,v_5).
\]
In Table~\ref{tab:bone-costs}, a tuple $(j_1,j_2,j_3,j_4,j_5)$ denotes the bijection $\psi:N(x)\to N(y)$ defined by $\psi(u_i)=v_{j_i}$ for every $i\in\{1,2,3,4,5\}$. The final column lists the distances $d(u_i,v_{j_i})$ for $i\in\{1,2,3,4,5\}$.

\begin{table}[htbp]
\centering
\caption{Assignments for the zero-idleness curvature of the five representative edge types of $H_q$.}
\label{tab:bone-costs}
\small
\setlength{\tabcolsep}{7pt}
\begin{tabular}{ccc}
\toprule
edge type & assignment $\psi$ & distances along $\psi$ \\
\midrule
$a_0a_1$ & $(5,2,3,4,1)$ & $(1,1,1,1,1)$ \\
$a_0a_q$ & $(2,3,1,5,4)$ & $(1,1,2,1,1)$ \\
$b_0b_1$ & $(1,5,3,4,2)$ & $(1,1,1,1,1)$ \\
$b_0b_{2q}$ & $(3,1,2,5,4)$ & $(2,1,1,1,1)$ \\
$a_0b_0$ & $(1,2,3,5,4)$ & $(1,1,2,1,1)$ \\
\bottomrule
\end{tabular}
\end{table}

From Table~\ref{tab:bone-costs}, we can see that the assignments for the two unit edge types have total cost $5$, and hence $m_0(a_0,a_1)=m_0(b_0,b_1)=5$. For the remaining three edge types, the displayed assignments have total cost $6$, so $m_0(a_0,a_q)\leq6$, $m_0(b_0,b_{2q})\leq6$, and $m_0(a_0,b_0)\leq6$.

We next show that none of the latter three edge types admits an assignment of total cost $5$. Let $\mathcal B^0_{xy}$ be the weighted complete bipartite graph with bipartition $(N(x),N(y))$, where each edge $zw$ with $z\in N(x)$ and $w\in N(y)$ has weight $w(zw)=d(z,w)$. Every bijection $\psi:N(x)\to N(y)$ corresponds uniquely to a perfect matching of $\mathcal B^0_{xy}$, and the total weight of this perfect matching is precisely $\sum_{z\in N(x)}d\bigl(z,\psi(z)\bigr)$. Consequently, $m_0(x,y)$ is the minimum weight of a perfect matching in $\mathcal B^0_{xy}$.

Since $N(x)\cap N(y)=\varnothing$, every edge of $\mathcal B^0_{xy}$ has weight at least one. As both parts have five vertices, a perfect matching of total weight $5$ exists if and only if there is a perfect matching all of whose edges have weight one. Let $\mathcal B^1_{xy}$ denote the spanning subgraph of $\mathcal B^0_{xy}$ consisting of the edges of weight one. Thus, $m_0(x,y)=5$ if and only if $\mathcal B^1_{xy}$ has a perfect matching.

For $I\subseteq N(x)$, let $N_{\mathcal B^1_{xy}}(I)$ denote the neighborhood of $I$ in $\mathcal B^1_{xy}$. For the edge type $a_0a_q$, take $I=\{a_{-q},b_0\}\subseteq N(a_0)$. Then $N_{\mathcal B^1_{a_0a_q}}(I)=\{a_0\}$, and hence $\lvert N_{\mathcal B^1_{a_0a_q}}(I)\rvert=1<2=\lvert I\rvert$. For the edge type $b_0b_{2q}$, take $I=\{a_0,b_{-2q}\}\subseteq N(b_0)$. Then $N_{\mathcal B^1_{b_0b_{2q}}}(I)=\{b_0\}$, so again $\lvert N_{\mathcal B^1_{b_0b_{2q}}}(I)\rvert=1<2=\lvert I\rvert$. Finally, for the edge type $a_0b_0$, take $I=\{a_{-q},a_q\}\subseteq N(a_0)$. Then $N_{\mathcal B^1_{a_0b_0}}(I)=\{a_0\}$, and thus $\lvert N_{\mathcal B^1_{a_0b_0}}(I)\rvert=1<2=\lvert I\rvert$. Therefore, Hall's condition fails in each of the three cases, and none of the corresponding graphs $\mathcal B^1_{xy}$ has a perfect matching. Hence none of these three edge types admits an assignment of total cost $5$. Combining this with the assignments of total cost $6$ displayed in Table~\ref{tab:bone-costs}, we obtain $m_0(a_0,a_q)=m_0(b_0,b_{2q})=m_0(a_0,b_0)=6$.

Since $H_q$ is $5$-regular, we have $\kzero(x,y)=1-m_0(x,y)/5$. Thus, $\kzero(x,y)=0$ for the two unit-edge types $a_0a_1$ and $b_0b_1$, whereas $\kzero(x,y)=-1/5$ for the edge types $a_0a_q$, $b_0b_{2q}$, and $a_0b_0$. Therefore, $H_q$ is Lin--Lu--Yau Ricci-flat but not $0$-Ricci-flat, and hence it is not bone-idle by Theorem~\ref{thm:bone-idle-characterization}. This completes the proof.
 \end{proof}

\section*{Acknowledgments}
This work was supported by the Natural Science Foundation of Shandong Province (Grant No. ZR2024MA073), the Institute for Basic Science (IBS-R029-C4), the National Natural Science Foundation of China (Grant No. 12171414), and the Taishan Scholars Special Project of Shandong Province.

\section*{Data availability statement}
No datasets were generated or analysed during the current study.  A short verification script for the finite assignment calculations is included with the submission files for editorial convenience.
\section*{Declaration of generative AI and AI-assisted technologies in the manuscript preparation process}
During the preparation of this work, the authors used AI for checking grammatical errors and typos during the revision stage. The authors reviewed and edited the output as needed and take full responsibility for the content of the published article.
\section*{Conflict of interest statement}
The authors declare that they have no conflict of interest.

\bibliographystyle{plainurl}

\end{document}